\PassOptionsToPackage{svgnames}{xcolor}

\documentclass[a4paper,11pt,fleqn,reqno,oneside]{amsart}

\usepackage{lmodern}

\usepackage{mathtools}

\usepackage{amssymb}
\usepackage{amsthm}
\usepackage{amsfonts}

\usepackage{mathrsfs} 
\usepackage[english]{babel}
\usepackage{upgreek}
\usepackage{esvect}
\usepackage{extarrows}
\usepackage[noadjust]{cite}

\let\Gamma\varGamma

\usepackage{xcolor}

\usepackage{todonotes}

\usepackage[
	paper = a4paper,
	top = 32mm, bottom = 32mm,
	left = 34mm, right = 34mm,
	footskip = 10mm, footnotesep=5mm, 
]
{geometry}

\usepackage{doi}

\usepackage{hyperref}
\hypersetup{
  pdftitle    = {Evolution of area-decreasing maps between two-dimensional Euclidean spaces},
  pdfsubject  = {},
  pdfauthor   = {Felix Lubbe},
  pdfkeywords = {Mean curvature flow, area-decreasing maps, Euclidean space},
  pdfcreator  = {},
  pdfproducer = {LaTeX with hyperref},
  hidelinks, 
  colorlinks,
  linkcolor=blue,
  citecolor=blue,
  urlcolor=RoyalBlue
}

\newcommand{\arxiv}[1]{\href{http://www.arxiv.org/abs/#1}{arXiv:#1}}

\usepackage[sorted-cites]{amsrefs}

\usepackage{etoolbox}
\patchcmd{\section}{\scshape}{\bfseries}{}{}
\makeatletter
\renewcommand{\@secnumfont}{\bfseries}
\patchcmd{\@startsection}
  {\@afterindenttrue}
  {\@afterindentfalse}
  {}{}
\makeatother

\newcommand{\Rb}{\mathbb{R}}

\newcommand{\Sym}{\mathrm{Sym}}

\DeclareMathOperator{\tr}{\mathrm{tr}}

\newcommand{\Ac}{\mathcal{A}}
\newcommand{\Bc}{\mathcal{B}}

\newcommand{\ee}{\mathrm{e}}

\newcommand{\<}{\langle}
\renewcommand{\>}{\rangle}

\newcommand{\der}{\mathrm{d}}
\DeclareMathOperator{\dd}{d\hspace{-2.5pt}}

\newcommand{\dt}{\partial_t}

\newcommand{\Lap}{\Delta}

\newcommand{\D}{\mathrm{D}\mkern 1mu}
\newcommand{\Dy}{\widetilde{\mathrm{D}}\mkern 1mu}

\newcommand{\gMetric}{\mathrm{g}}
\newcommand{\gind}{\gMetric}

\newcommand{\tgind}{\widetilde{\gind}}

\newcommand{\sTensor}{\mathrm{s}}
\newcommand{\sind}{\sTensor}

\newcommand{\RpiM}{\pi_1}
\newcommand{\RpiN}{\pi_2}
\newcommand{\piM}{\RpiM}
\newcommand{\piN}{\RpiN}

\newcommand{\Rm}{\mathrm{R}}
\newcommand{\Rind}{\Rm}
\DeclareMathOperator{\Ric}{Ric}

\newcommand{\A}{\mathrm{A}}
\newcommand{\Hv}{\vv{H}} 			

\newcommand{\RM}{\Rb^2}
\newcommand{\RN}{\Rb^2}

\newcommand{\M}{\RM}
\newcommand{\N}{\RN}

\newcommand{\RgM}{\gMetric_{\RM}}
\newcommand{\RgN}{\gMetric_{\RN}}
\newcommand{\RgMN}{\gMetric_{\RM\times\RN}}
\newcommand{\RsMN}{\sTensor_{\RM\times\RN}}
\newcommand{\sMN}{\RsMN}

\newcommand{\gM}{\RgM}
\newcommand{\gN}{\RgN}

\newcommand{\nN}{\nabla^{\perp}}
\newcommand{\prN}{\mathrm{pr}^{\perp}}

\newcommand{\sv}{\lambda} 

\DeclareSymbolFont{UPM}{U}{eur}{m}{n}
\DeclareMathSymbol{\partial}{0}{UPM}{"40}

\newcommand{\fScaled}[1]{\widetilde{f}_{#1}}
\newcommand{\tScaled}{r}
\newcommand{\xScaled}{y}
\newcommand{\fys}[1]{\fScaled{#1}(\xScaled,\tScaled)}

\usepackage[shortlabels]{enumitem}
\setlist[enumerate,1]{leftmargin=1cm}

\newtheorem{mainthm}{Theorem}
\newtheorem{theorem}{Theorem}
\newtheorem{lemma}[theorem]{Lemma}
\newtheorem{corollary}[theorem]{Corollary}

\theoremstyle{definition}
\newtheorem{remark}[theorem]{Remark}
\newtheorem{definition}[theorem]{Definition}
\newtheorem{example}[theorem]{Example}

\numberwithin{equation}{section}
\numberwithin{theorem}{section}

\allowdisplaybreaks

\title[Evolution of area-decreasing maps]{Evolution of area-decreasing maps between two-dimensional Euclidean spaces}
\author{Felix Lubbe}
\date{}
\keywords{Mean curvature flow, area-decreasing maps, Euclidean space}
\subjclass[2010]{Primary 53C44; 53C42, 53A07}

\begin{document}

\begin{abstract}
	We consider the mean curvature flow of the graph of a
	smooth map $f:\Rb^2\to\Rb^2$ between two-dimensional Euclidean spaces.
	If $f$ satisfies an area-decreasing property, 
	the solution exists for all times and
	the evolving
	submanifold stays the graph of an area-decreasing map $f_t$. 
	Further, we prove uniform decay estimates for
	the mean curvature vector of the graph and all higher-order derivatives
	of the corresponding map $f_t$.
\end{abstract}

\maketitle

\section{Introduction}

Let $(M,\gMetric_M)$ and $(N,\gMetric_N)$ be complete Riemannian manifolds, 
and consider a 
smooth 
map $f:M\to N$.
Then $f$ is called \emph{strictly length-decreasing}, if
there is $\delta\in(0,1]$, such that $\|\dd f(v)\|_{\gMetric_N}\le(1-\delta)\|v\|_{\gMetric_M}$
for all $v\in\Gamma(TM)$.
The map $f$ is called 
\emph{strictly area-decreasing} if there is $\delta\in(0,1]$, such that 
\[
	\bigl\| \dd f(v) \wedge \dd f(w) \bigr\|_{\gMetric_N} \le (1-\delta) \| v \wedge w \|_{\gMetric_M}
\]
for all $v,w\in \Gamma(TM)$. In this paper, we 
deform the map $f$ by deforming its corresponding graph
\[
	\Gamma(f) \coloneqq \big\{ (x,f(x)) \in M\times N : x\in M \big\}
\]
via the mean curvature flow in the product
space $M\times N$. That is, we consider 
the system
\[
	\partial_t F_t(x) = \Hv(x,t) \,, \qquad F_0(x) = \bigl( x, f(x) \bigr) \,,
\]
where $\Hv(x,t)$ denotes the mean curvature vector of the submanifold $F_t(M)$ in $M\times N$
at $F_t(x)$. A smooth solution to the mean curvature flow for which $F_t(M)$ is a graph for $t\in[0,T_g)$
can be described completely in terms of a smooth
family of maps $f_t:M\to N$ with $f_0=f$,
where $0<T_g\le \infty$ is the maximal time for which the graphical solution exists.
In the case of long-time existence of the graphical solution (i.\,e.\ $T_g=\infty$) and convergence, we 
obtain a smooth homotopy from $f$ to a minimal map $f_{\infty}:M\to N$. 
Recall that a map between $M$ and $N$ is called \emph{minimal}, if its graph is a minimal submanifold
of the product space $M\times N$ \cite{Sch93}. \\

For a compact domain and arbitrary dimensions, several results for length- and area-decreasing maps are known
(see e.\,g.\ \cites{LL11,SHS13,SHS14,SHS15,Smo04,Wan01a,Wan01b,Wan01c} and references therein). 
For example,
if $f:M\to N$ is strictly area-decreasing,
$M$ and $N$ are space forms with $\dim M\ge 2$,
and their sectional curvatures satisfy 
\[
	\sec_M \ge |\sec_N| \,, \qquad \sec_M + \sec_N > 0 \,,
\]
Wang and Tsui proved long-time existence of the graphical mean curvature flow
and convergence of $f$ to a constant map \cite{TW04}. Subsequently, the curvature assumptions
on the manifolds were relaxed by Lee and Lee \cite{LL11} and recently by
Savas-Halilaj and Smoczyk \cite{SHS14}.

In the non-compact setting, Ecker and Huisken considered the flow of entire graphs, that is,
graphs generated by maps $f:\Rb^n\to\Rb$. The quantity which plays an
important role is essentially given
by the Jacobian of the projection map from the graph $\Gamma(f)$ to $\Rb^n$ and it satisfies a nice evolution equation.
They provided conditions under which the mean curvature flow 
of the graph exists for all time and asymptotically approaches self-expanding solutions.
\cites{EH89,EH91}.
Unfortunately, their methods cannot easily be adapted to the general higher-codimensional setting, since 
the analysis gets considerably more involved due to the complexity of the normal bundle of the graph.

Nevertheless, several results in the higher-codimensional case were obtained by
considering the Gau{\ss} map of the immersion (see e.\,g.\ \cites{Wan03b,Wan05}).
In the case of two-dimensional graphs, Chen, Li and Tian established long-time existence and convergence
results by evaluating certain angle functions on the tangent bundle \cite{CLT02}. 
Another possibility is to impose suitable smallness conditions on the differential of the defining map.
In these cases, one can show long-time existence and convergence of the mean curvature flow \cites{CCH12,CCY13,SHS15}.

Considering maps between Euclidean spaces of the same dimension,
Chau, Chen and He obtained results for strictly length-decreasing 
Lipschitz continuous maps
$f:\Rb^m\to\Rb^m$ with graphs $\Gamma(f)$ being
Lagrangian submanifolds of $\Rb^m\times\Rb^m$.
In particular, they showed short-time existence of
solutions with bounded geometry, as well as decay estimates for the mean
curvature vector and all higher-order derivatives of the defining map, which
in turn imply the long-time existence of the solution
\cite{CCH12}. Recently, the estimates were generalized by the author to Euclidean spaces of arbitrary dimension \cite{Lub16}. \\

The aim of the article at hand is to relax the length-decreasing 
assumption for maps between some Euclidean spaces.
Namely, if both domain and target are
of dimension two, we are able to show the following result.

\begin{mainthm}
	\label{thm:ThmA}
	Suppose $f:\Rb^2\to\Rb^2$ is a smooth strictly area-decreasing  
	function.
	Then the mean curvature flow with initial condition $F(x)\coloneqq\bigl(x,f(x)\bigr)$
	has a long-time smooth solution for all $t>0$
	such that the following statements hold.
	\begin{enumerate}
		\item Along the flow, the evolving surface stays the graph of a strictly area-decreasing
			map $f_t:\Rb^2\to\Rb^2$ for all $t>0$.
		\item The mean curvature vector of the graph satisfies the estimate
			\[
				t\|\Hv\|^2 \le C
			\]
			for some constant $C\ge 0$.
		\item All spatial derivatives of $f_t$ of order $k\ge 2$ satisfy the estimate
			\[
				t^{k-1} \sup_{x\in\Rb^2} \|\D^k f_t(x)\|^2 \le C_{k,\delta} \quad \text{for all} \quad k\ge 2
			\]
			and for some constants $C_{k,\delta}\ge 0$ depending only on $k$ and $\delta$. Moreover,
			\[
				\sup_{x\in\Rb^2} \| f_t(x)\|^2 \le \sup_{x\in\Rb^2} \| f(x) \|^2
			\]
			for all $t>0$.
	\end{enumerate}
	If in addition $f$ satisfies $\|f(x)\|\to 0$ as $\|x\|\to\infty$, then $\|f_t(x)\|\to 0$ smoothly
	on compact subsets of $\Rb^2$ as $t\to\infty$.
\end{mainthm}

\begin{remark}
			In terms of the second fundamental form of the graph,
			Theorem \ref{thm:ThmA} implies
			the decay estimate
			\[
				t \|\A\|^2 \le C
			\]
			for some constant $C\ge 0$ depending only on $\delta$.
\end{remark}

\begin{remark}
	\begin{enumerate}
		\item Note that any 
			strictly length-decreasing map is also strictly area-decreasing.
			Accordingly, for 
			smooth maps between two-dimensional Euclidean spaces
			the statement of \cite{Lub16}*{Theorem A}
			follows from Theorem \ref{thm:ThmA}.
		\item In the recent paper \cite{SHS16}, the case of area-decreasing maps
			between complete Riemann surfaces with bounded geometry $M$ and $N$ is treated, where $M$ is compact
			and the sectional curvatures satisfy $\min_{x\in M} \sec_M(x) \ge \sup_{x\in N} \sec_N(x)$.
	\end{enumerate}
\end{remark}

\begin{remark}
	If one considers graphs generated by functions $f:\Rb^2\to\Rb$, the same
	strategy as in the proof of Theorem \ref{thm:ThmA} can be applied. 
	To draw the conclusions of the theorem in this case, one only has to assume
	that the differential $\dd f$ of $f$ is bounded, 
	i.\,e.\ there is a constant $C\ge 0$ with
	\[
		|\dd f(v)| \le C \|v\|
	\]
	for any $v\in\Gamma(T\Rb^2)$.
	In particular,  
	the function has at most linear growth, so that it belongs to the class of functions studied in \cite{EH89}.
\end{remark}

The outline of the paper is as follows. 
In Section \ref{sec:Maps}, we 
introduce the main quantities in the graphical case which then will be deformed by
the mean curvature flow described in Section \ref{sec:MCF}.
To obtain the statements of the following sections, we would like to apply
a maximum principle. For this, we follow an idea from \cite{CCH12} to 
adapt the usual scalar maximum principle to the non-compact case.
Then, in Section \ref{sec:Pres} we establish the preservation of
the area-decreasing condition. In Section \ref{sec:APriori} we obtain estimates on $\Hv$
and all derivatives of the map defining the graph by considering functions
constructed similar to those in \cite{STW16}. 
The main theorem is proven in Section \ref{sec:Proof} and some applications 
to self-similar solutions of the mean curvature flow are given in Section \ref{sec:Appl}.

\subsection*{Acknowledgments}
The author would like to thank Andreas Savas-Halilaj for stimulating discussions.
		
\section{Maps between two-dimensional Euclidean spaces}
\label{sec:Maps}

\subsection{Geometry of graphs}
\label{sec:GeomOfGraphs}
We recall the geometric quantities in a graphical
setting adopted to two-dimensional Euclidean spaces. For the setup in generic
Euclidean spaces, see e.\,g.\ \cite{Lub16}*{Section 2} and for the general setup, see e.\,g.\ \cite{SHS15}*{Section 2}.

Let $(\M,\gM)$ be the two-dimensional Euclidean space equipped with its usual flat metric.
On the product manifold $(\M\times \N,\<\cdot,\cdot\>\coloneqq\gM\times\gN)$,
the projections onto the first and second factor
\[
	\RpiM, \RpiN : \M\times \N \to \M
\]
are submersions, that is they are smooth and have maximal rank. 
A smooth map $f:\M\to \N$ defines an embedding $F:\M\to \M\times \N$ via
\[
	F(x) \coloneqq \bigl(x,f(x)\bigr)\,,\qquad x\in \M\,.
\]
The \emph{graph of $f$} is defined to be the submanifold 
\[
	\Gamma(f) \coloneqq F(\M) = \left\{ \bigl(x,f(x)\bigr) : x\in \M \right\} \subset \M\times \N \,.
\]
Since $F$ is an embedding, it induces another Riemannian metric on $\M$, given by
\[
	\gind \coloneqq F^{*}\<\cdot,\cdot\> \,.
\]
The metrics $\gM,\<\cdot,\cdot\>$ and $\gind$ are related by
\begin{align*}
	\<\cdot,\cdot\> &= \piM^{*} \gM + \piN^{*}\gN \,, \\
	\gind &= F^{*}\<\cdot,\cdot\> = \gM + f^{*}\gN \,.
\end{align*}
As in \cites{SHS13,SHS14}, let us introduce the symmetric $2$-tensors
\begin{align*}
	\sMN &\coloneqq \piM^{*}\gM - \piN^{*}\gN\,, \\
	\sind &\coloneqq F^{*}\sMN = \gM - f^{*}\gN \,.
\end{align*}
We remark that $\sMN$ is a semi-Riemannian metric of signature $(2,2)$ on
$\M\times \N$. \\

The Levi-Civita connection on $\M$ with respect to the induced
metric $\gind$ is denoted by $\nabla$ and the corresponding curvature
tensor by $\Rind$. 

\subsection{Second fundamental form}

The \emph{second fundamental tensor} of the graph $\Gamma(f)$ is the section 
$\A\in\Gamma\bigl(T^{\perp}\M \otimes \Sym(T^*\M\otimes T^*\M)\bigr)$
defined as
\[
	\A(v,w) \coloneqq \bigl( \nabla \dd F\bigr)(v,w) \coloneqq \D_{\dd F(v)} \dd F(w) - \dd F(\nabla_v w) \,,
\]
where $v,w\in \Gamma(T\M)$ and where we denote the connection on 
$F^{*}T(\M\times \N)\otimes T^{*}\M$ induced by the Levi-Civita connection also by $\nabla$.
The trace of $\A$ with respect to the metric $\gind$ is called the
\emph{mean curvature vector field} of $\Gamma(f)$ and it will be denoted by
\[
	\Hv \coloneqq \tr \A \,.
\]
Let us denote the evaluation of the second fundamental
form (resp.\ mean curvature vector) in the direction of a vector $\xi\in\Gamma\bigl(F^*T(\M\times\N)\bigr)$ by
\[
	\A_{\xi}(v,w) \coloneqq \bigl\<\A(v,w),\xi\bigr\> \qquad \text{resp.} \qquad \Hv_{\xi} \coloneqq \bigl\<\Hv, \xi\bigr\> \,.
\]
Note that $\Hv$ is a section in the normal bundle of the graph.
If $\Hv$ vanishes identically, the graph is said to be minimal.
A smooth map $f:\M\to \N$ is called \emph{minimal},
if its graph $\Gamma(f)$ is a minimal submanifold of the product
space $(\M\times \N,\<\cdot,\cdot\>)$. \\

On the submanifold, the \emph{Gau{\ss} equation}
\begin{equation}
	\label{eq:Gauss}
	\Rind(u_1,v_1,u_2,v_2) = \bigl\< \A(u_1,u_2), \A(v_1,v_2) \bigr\> - \bigl\< \A(u_1,v_2), \A(v_1,u_2) \bigr\>
\end{equation}
and the \emph{Codazzi equation}
\[
	(\nabla_u \A)(v,w) - (\nabla_v \A)(u,w) = - \dd F\bigl(\Rind(u,v)w\bigr)
\]
hold, where the induced connection on the bundle $F^*T(\M\times \N)\otimes T^*\M \otimes T^*\M$
is defined as
\[
	(\nabla_u \A)(v,w) \coloneqq \D_{\dd F(u)}(\A(v,w)) - \A(\nabla_uv,w) - \A(v,\nabla_uw) \,.
\]

\subsection{Singular value decomposition}
\label{sec:SVD}

We recall the singular value decomposition theorem for the two-dimensional case (see e.\,g.\ \cite{SHS13}*{Section 3.2} for the general setup). \\

Fix a point $x\in \M$, and let
\[
	\sv_1^2(x) \le \sv_2^2(x)
\]
be the eigenvalues of $f^{*}\gN$ with respect to $\gM$. The
values
$0\le \sv_1(x)\le \sv_2(x)$ are called the
\emph{singular values} of the differential $\dd f$ of $f$ and give rise
to continuous functions on $\M$. 
At the
point $x$ consider an orthonormal basis $\{\alpha_{1},\alpha_{2}\}$
with respect to $\gM$ which diagonalizes $f^{*}\gN$. Moreover, at $f(x)$ 
consider a basis $\{\beta_{1},\beta_{2}\}$ that is
orthonormal
with respect to $\gN$, such that
\[
	\dd f(\alpha_{1}) = \sv_1(x) \beta_{1} \,, \qquad \dd f(\alpha_{2}) = \sv_2(x) \beta_{2}
\]
This procedure is called the \emph{singular value decomposition} of the differential $\dd f$. \\

Now let us construct a special basis for the tangent and the normal
space of the graph in terms of the singular values. The vectors
\[
	\widetilde{e}_{1} \coloneqq \frac{1}{\sqrt{1+\sv_{1}^{2}(x)}}\bigl( \alpha_{1} \oplus \sv_{1}(x)\beta_{1} \bigr) \quad \text{and} \quad \widetilde{e}_{2} \coloneqq \frac{1}{\sqrt{1+\sv_{2}^{2}(x)}}\bigl( \alpha_{2} \oplus \sv_{2}(x)\beta_{2} \bigr) 
\]
form an orthonormal basis with respect to the metric $\<\cdot,\cdot\>$ of 
the tangent space $\dd F(T_{x}\M)$ of the graph $\Gamma(f)$ at $x$. It follows
that with respect to the induced metric $\gind$, the vectors
\[
	e_1 \coloneqq \frac{1}{\sqrt{1+\sv_1^2(x)}} \alpha_1 \qquad \text{and} \qquad e_2 \coloneqq \frac{1}{\sqrt{1+\sv_2^2(x)}} \alpha_2 
\]
form an orthonormal basis of $T_x\M$.
Moreover,
the vectors
\[
	\xi_{1} \coloneqq \frac{1}{\sqrt{1+\sv_{1}^{2}(x)}}\bigl(-\sv_{1}(x)\alpha_{1} \oplus \beta_{1}\bigr) \quad \text{and} \quad
	\xi_{2} \coloneqq \frac{1}{\sqrt{1+\sv_{2}^{2}(x)}}\bigl(-\sv_{2}(x)\alpha_{2} \oplus \beta_{2}\bigr)
\]
form an orthonormal basis with respect to $\<\cdot,\cdot\>$ of the
normal space $T^{\perp}_{x}\M$ of the graph $\Gamma(f)$ at the point
$x$. From the formulae above, we deduce that
\[
	\sMN\bigl(\widetilde{e}_{i},\widetilde{e}_{j}\bigr) = \sind(e_i,e_j) = \frac{1-\sv_{i}^{2}(x)}{1+\sv_{i}^{2}(x)}\delta_{ij} \,, \qquad 1\le i,j\le 2 \,.
\]
Therefore, the eigenvalues of the $2$-tensor $\sind$ with respect to $\gind$
are given by
\begin{equation}
	\label{eq:sSV}
	\frac{1-\sv_{1}^{2}(x)}{1+\sv_{1}^{2}(x)} \ge \frac{1-\sv_{2}^{2}(x)}{1+\sv_{2}^{2}(x)} \,.
\end{equation}
Moreover,
\begin{equation}
	\label{eq:sOnNormalBdl}
	\sMN(\xi_{i},\xi_{j}) = - \frac{1-\sv_{i}^{2}(x)}{1+\sv_{i}^{2}(x)}\delta_{ij} \,,\qquad 1\le i,j\le 2\,,
\end{equation}
and
\[
	\sMN(\widetilde{e}_{i},\xi_{j}) = - \frac{2\sv_{i}(x)}{1+\sv_{i}^{2}(x)} \delta_{ij} \,,\qquad 1\le i,j\le 2 \,.
\]

\section{Mean curvature flow in Euclidean space}
\label{sec:MCF}

Let $I\coloneqq[0,T)$ for some $T>0$ and assume $f_0:\RM\to \RN$ to be a smooth map. We say that a family
of maps $F:\RM\times I \to \RM\times \RN$ evolves under the mean curvature flow, if for all $x\in \RM$ 
\begin{equation}
	\label{eq:MCF}
	\begin{cases} \partial_t F(x,t) = \Hv(x,t) \,, \\ F(x,0) = \bigl(x,f_0(x)\bigr) \,. \end{cases}
\end{equation}
This system can also be described as follows. As in \cite{CCH12}*{Section 5}, let us consider the non-parametric
mean curvature flow equation for $f:\RM\times[0,T)\to\RN$, given by the quasilinear system
\begin{equation}
	\label{eq:pMCF}
	\begin{cases} \partial_t f(x,t) = \sum_{i,j=1}^2 \tgind^{ij} \partial^2_{ij} f(x,t) \,, \\ f(x,0) = f_0(x) \,, \end{cases}
\end{equation}
where $\tgind^{ij}$ are the components of the inverse of $\tgind \coloneqq \RgM + f_t^*\RgN$, where
here we have set $f_t(x)\coloneqq f(x,t)$.
If
\eqref{eq:pMCF} has a smooth solution $f:\RM\times[0,T)\to\RN$, then the mean curvature flow \eqref{eq:MCF}
has a smooth solution $F:\RM\times[0,T)\to\RM\times\RN$ given by the family of graphs
\[
	\Gamma\bigl(f(\cdot,t)\bigr) = \bigl\{ \bigl(x,f(x,t)\bigr) : x\in\RM \bigr\} \,,
\]
up to tangential diffeomorphisms (see e.\,g.\ \cite{Bra78}*{Chapter 3.1}).

In the sequel, if there is no confusion, we will also use the notation $F_t(x)\coloneqq F(x,t)$ as
well as $f_t(x)\coloneqq f(x,t)$. \\

For \eqref{eq:pMCF}, we have the following short-time existence result.

\begin{theorem}[{\cite{CCH12}*{Proposition 5.1 for $m=n=2$}}]
	\label{thm:ShortTimeEx}
	Suppose $f_0:\RM\to\RN$ is a smooth function, such that for each $l\ge 1$ we have 
	$\sup_{x\in\RM}\|\D^l f_0(x)\|\le C_l$ for some finite constants $C_l$.
	Then \eqref{eq:pMCF} has a short-time smooth solution $f$ on $\RM\times[0,T)$
	for some $T>0$ with initial condition $f_0$,
	such that $\sup_{x\in\RM}\|\D^lf_t(x)\| < \infty$ for every $l\ge 1$ and $t\in[0,T)$.
\end{theorem}

In this paper, we will consider a special kind of solution to \eqref{eq:MCF}.
\begin{definition}
	Let $F_t(x)$ be a smooth solution to the system \eqref{eq:MCF} on $\RM\times[0,T)$ for some
	$0<T\le\infty$, such that for each $t\in[0,T)$ and non-negative integer $k$, 
	the submanifold $F_t(\RM)\subset\RM\times\RN$ satisfies
	\begin{gather}
		\label{eq:BddGeom1}
		\sup_{x\in\RM} \|\nabla^k \A(x,t) \| < \infty \,, \\
		\label{eq:BddGeom2}
		C_1(t) \RgM \le \gind \le C_2(t) \RgM \,,
	\end{gather}
	where $C_1(t)$ and $C_2(t)$ for each $t\in[0,T)$ are finite, positive constants
	depending only on $t$. Then we will say that the family of embeddings 
	$\{F_t\}_{t\in I}$
	has \emph{bounded geometry}.
\end{definition}

\begin{definition}
	Let $f_t(x)$ be a smooth solution to the system \eqref{eq:pMCF} on $\RM\times[0,T)$ for
	some $0<T\le\infty$, such that for each $t\in[0,T)$ and positive integer $k$ the estimate 
	\[
		\sup_{x\in\RM} \| \D^k f_t(x) \| < \infty
	\]
	holds. Then we will say that $f_t(x)$ has \emph{bounded geometry} for every $t\in[0,T)$.
\end{definition}

\subsection{Graphs}
We recall some important notions in the graphic case, where we follow the presentation in \cite{SHS14}*{Section 3.1}.

Let $f_0:\RM\to\RN$ denote a smooth map, such that 
$ \sup_{x\in\RM}\| \D^l f_0(x) \| \le C_l$ for some finite constants $C_l$, $l\ge 1$.
Then Theorem \ref{thm:ShortTimeEx} ensures that the system
\eqref{eq:pMCF} has a short-time solution with initial data $f_0(x)$
on a time interval $[0,T)$ for some positive maximal time $T>0$. Further, there
is a diffeomorphism $\phi_t:\RM\to\RM$, such that 
\begin{equation}
	\label{eq:DiffeoRel}
	F_t\circ \phi_t(x)=(x,f_t(x)) \,,
\end{equation}
where $F_t(x)$ is a solution of \eqref{eq:MCF}.

To obtain the converse of this statement,
let $\Omega_{\RM}$ be the volume form on $\RM$ and extend it to a parallel $2$-form on $\RM\times\RN$
by pulling it back via the natural projection onto the first factor $\RpiM:\RM\times\RN\to\RM$, that is, consider
the $2$-form $\RpiM^*\Omega_{\RM}$. Define the time-dependent smooth function $u:\RM\times[0,T)\to\Rb$,
given by
\[
	u \coloneqq \star \Omega_t \,,
\]
where $\star$ is the Hodge star operator with respect to the induced metric $\gind$ and
\[
	\Omega_t \coloneqq F_t^*\bigl( \RpiM^*\Omega_{\RM} \bigr) = (\RpiM \circ F_t)^* \Omega_{\RM}\,.
\]
The function $u$ is the Jacobian of the projection map from $F_t(\RM)$ to $\RM$. From
the implicit mapping theorem it follows that $u>0$ if and only if there is a diffeomorphism
$\phi_t:\RM\to\RM$ and a map $f_t:\RM\to\RN$, such that \eqref{eq:DiffeoRel} holds,
i.\,e.\ $u$ is positive precisely if the solution of the mean curvature flow remains a graph.
By Theorem \ref{thm:ShortTimeEx}, 
the solution will stay a graph at least in a short time
interval $[0,T)$.

\subsection{Parabolic scaling}

For any $\tau>0$ and $(x_0,t_0)\in\RM\times[0,T)$, consider the change of variables
\[
	\xScaled \coloneqq \tau (x-x_0) \,, \qquad \tScaled \coloneqq \tau^2(t-t_0) \,, \qquad \fys{\tau} \coloneqq \tau \bigl( f(x,t) - f(x_0,t_0) \bigr) \,,
\]
which we call the \emph{parabolic scaling by $\tau$ at $(x_0,t_0)$}. Let us denote the derivative with
respect to $y$ by $\Dy$, and let $\tgind_{\tau}$ be the scaled metric as well as
$\A_{\tau}$ be the scaled second fundamental form. We calculate
\[
	(\Dy^k \fScaled{\tau})(\xScaled,\tScaled) = \tau^{1-k}(\D^kf)(x,t) \,,
\]
which implies
\[
	\tgind_{\tau\,|(\xScaled,\tScaled)} = \tgind_{|(x,t)} \qquad \text{and} \qquad \A_{\tau\,|(\xScaled,\tScaled)} = \frac{1}{\tau} \A_{|(x,t)} \,, 
\]
so that $\fys{\tau}$ satisfies equation \eqref{eq:pMCF} in the sense that
\[
	\frac{\partial \fScaled{\tau}}{\partial \tScaled}(\xScaled,\tScaled) = \sum_{i,j=1}^2 \tgind_{\tau}^{ij} \frac{\partial^2 \fScaled{\tau}}{\partial \xScaled^i \partial \xScaled^j}(\xScaled,\tScaled) \,.
\]

\subsection{Evolution equations}

Let us recall the evolution equation of the tensor $\sind$ in the two-dimensional
setting (which is basically calculated in \cite{SHS14}*{Lemma 3.1}),
as well as the evolution equation for its trace.

\begin{lemma}
	\label{lem:sEvolEq}
	Under the mean curvature flow, the evolution of the tensor $\sind$ for
	$t\in[0,T)$ is given by the formula
	\begin{align*}
		\left( \nabla_{\dt} \sind - \Lap \sind\right)(v,w) = & - \sind (\Ric v, w ) - \sind( v, \Ric w ) \\
			& - 2 \sum_{k=1}^2 \RsMN\bigl( \A(e_k,v), \A(e_k,w) \bigr) \,,
	\end{align*}
	where $\{e_1,e_2\}$ is any orthonormal frame with respect to $\gind$
	and where the \emph{Ricci operator} is given by
	\[
		\Ric v \coloneqq - \sum_{k=1}^2 \Rind(e_k,v) e_k \,.
	\]
\end{lemma}

\begin{corollary}
	\label{cor:TrEvol}
	Under the mean curvature flow, the evolution equation of the trace of the tensor $\sind$
	is given by
	\[
		\left( \dt - \Lap \right) \tr(\sind) = - 2 \sum_{k,l=1}^2 \left( \RsMN - \frac{1-\sv_k^2}{1+\sv_k^2} \RgMN \right)\bigl( \A(e_k,e_l), \A(e_k,e_l) \bigr) \,,
	\]
	where $\{e_1,e_2\}$ denotes the orthonormal frame field with respect to $\gind$
	constructed in Section \ref{sec:SVD}.
\end{corollary}

\begin{proof}
	From the Gau{\ss} Equation \eqref{eq:Gauss} we obtain
	\begin{align*}
		\sind (\Ric e_k, e_k) ={} &- \sind(e_k,e_k) \sum_{l=1}^2 \RgMN\bigl( \A(e_k,e_l), \A(e_k,e_l) \bigr) \\
			& + \sind(e_k,e_k) \RgMN\bigl( \Hv, \A(e_k,e_k) \bigr) \,.
	\end{align*}
	Further, since
	\[
		\dt \tr(\sind) = 2 \sum_{k=1}^2 \RgMN\bigl( \Hv, \A(e_k,e_k) \bigr) \sind(e_k,e_k) + \sum_{k=1}^2 (\nabla_{\dt}\sind)(e_k,e_k) \,,
	\]
	the claim follows from Lemma \ref{lem:sEvolEq}.
\end{proof}

In the two-dimensional setting at hand, we can rewrite the evolution equation for the trace.

\begin{lemma}
	\label{lem:TrEvol}	
	Under the mean curvature flow, the trace of the tensor $\sind$ satisfies
	\begin{align*}
		\left( \dt - \Lap \right) \tr(\sind) = 2 \|\A\|^2 \tr(\sind) &- \frac{1}{2} \frac{\|\nabla\tr(\sind)\|^2}{\tr(\sind)} \\
			& + \frac{2}{\tr(\sind)} \sum_{k=1}^2 \left( \frac{2\sv_2}{1+\sv_2^2} \A_{1k}^1 + \frac{2\sv_1}{1+\sv_1^2} \A_{2k}^2 \right)^2 \,.
	\end{align*}
\end{lemma}

\begin{proof}
	This is \cite{STW16}*{Eqs.\ (3.17) and (3.18)}.
\end{proof}

\section{Evolution of submanifold geometry}

\subsection{Preserved quantities}
\label{sec:Pres}

Consider a smooth map $f:\Rb^2\to\Rb^2$. The property of $f$ being strictly area-decreasing
can be expressed in terms of the singular values $\sv_1,\sv_2$ of the differential $\dd f$ as
\[
	\sv_1^2 \sv_2^2 \le 1-\delta \,. 
\]
for some $\delta\in(0,1]$.
Consequently, by Equation \eqref{eq:sSV}, this can also be rephrased in terms of the tensor $\sind$ as follows.
If $f$ is strictly area-decreasing, there is $\varepsilon>0$, such that
the inequality
\[
	\tr(\sind) = \frac{2(1-\sv_1^2\sv_2^2)}{(1+\sv_1^2)(1+\sv_2^2)} \ge \varepsilon
\]
holds. We will now modify $\tr(\sind)-\varepsilon$ using
the function
\begin{equation}
	\label{eq:PhiDef}
	\phi_R(x) \coloneqq 1 + \frac{\|x\|^2_{\Rb^2}}{R^2} \,,
\end{equation}
where $\|\cdot\|_{\Rb^2}$ is the Euclidean norm on $\Rb^2$ and $R>0$ is a constant
which will be chosen later. 

\begin{lemma}
	\label{lem:phiEst}
	Let $F(x,t)$ be a smooth solution to \eqref{eq:MCF} with bounded geometry and assume
	there is $\varepsilon>0$, such that $\tr(\sind) \ge \varepsilon$ for any $t\in[0,T)$.
	Fix any $T'\in[0,T)$ and $(x_0,t_0)\in\Rb^2\times[0,T']$. Then the following estimates hold,
	\begin{align*}
		- c(T') \frac{\|x_0\|_{\Rb^2}}{R^2} \tr(\sind) &\le \< \nabla \phi_R, \nabla \tr(\sind) \> \le c(T') \frac{\|x_0\|_{\Rb^2}}{R^2} \tr(\sind) \,, \\
		| \Lap \phi_R | &\le c(T') \left( \frac{1}{R^2} + \frac{\|x_0\|_{\Rb^2}}{R^2} \right) \,,
	\end{align*}
	where $c(T')\ge 0$ is a constant depending only on $T'$.
\end{lemma}

\begin{proof}
	Note that
	\[
		\nabla_u \tr(\sind) = \sum_{k=1}^2 (\nabla_u\sind)(e_k,e_k) = 2 \sum_{k=1}^2 \sind_{\Rb^2\times\Rb^2}\bigl( \A(u,e_k), \dd F(e_k) \bigr) \,.
	\]
	The bounded geometry assumptions \eqref{eq:BddGeom1} and \eqref{eq:BddGeom2} imply that $\sind$, $\nabla \sind$
	and therefore $\nabla \tr(\sind)$ are uniformly bounded on $\Rb^2\times[0,T']$ by a constant $c(T')$ depending only on $T'$.
	Thus, also using $\tr(\sind)\ge \varepsilon$, at $(x_0,t_0)$ we have
	\[
		-c(T') \frac{\|x_0\|_{\Rb^2}}{R^2} \tr(\sind) \le \< \nabla \phi_R, \nabla \tr(\sind) \> \le c(T') \frac{\|x_0\|_{\Rb^2}}{R^2} \tr(\sind) \,.
	\]
	
	The statement for $|\Lap\phi_R|$ is given in \cite{CCH12}*{Eq.\ 3.4}.
\end{proof}

Let us define
\[
	\uppsi(x,t) \coloneqq \ee^{\sigma t} \phi_R(x) \tr(\sind)_{|(x,t)} - \varepsilon \,.
\]

\begin{lemma}
	\label{lem:AreaDecrEpsilon}
	Let $F(x,t)$ be a smooth solution to \eqref{eq:MCF} with bounded geometry. Assume there
	is $\varepsilon>0$ with $\tr(\sind)\ge \varepsilon$ at $t=0$, and
	$\tr(\sind)\ge \frac{\varepsilon}{2}$ for all $t\in[0,T)$. Then it is
	$\tr(\sind)\ge \varepsilon$ for all $t\in[0,T)$.
\end{lemma}

\begin{proof}
	The proof closely follows the strategy in \cites{CCH12,Lub16}.
	We will show that for any fixed $T'\in[0,T)$ and $\sigma>0$, there is $R_0>0$ depending
	only on $\sigma$ and $T'$, such that $\uppsi>0$ on $\Rb^2\times[0,T']$ for all
	$R\ge R_0$.
	
	On the contrary, suppose $\uppsi$ is not positive on $\Rb^2\times[0,T']$
	for some $R\ge R_0$. Then as $\psi>0$ on $\Rb^2\times\{0\}$, $\tr(\sind)\ge \frac{\varepsilon}{2}$
	on $\Rb^2\times[0,T)$ and $\phi_R(x)\to\infty$ as $\|x\|\to\infty$, it follows
	that $\uppsi>0$ outside some compact set $K\subset\Rb^2$ for all $t\in[0,T)$.
	We conclude that there is $(x_0,t_0)\in K\times[0,T']$ such that 
	$\uppsi(x_0,t_0)=0$ at $(x_0,t_0)$ and that $t_0$ is the first such time.
	According to the second derivative criterion, at the point $(x_0,t_0)$
	we have
	\begin{equation}
		\label{eq:PsiCond}
		\partial_t \uppsi \le 0 \,, \qquad \nabla \uppsi = 0 \qquad \text{and} \qquad \Lap \uppsi \ge 0 \,.
	\end{equation}
	On the other hand, using Lemma \ref{lem:TrEvol}, we estimate the terms in the evolution equation for $\uppsi$,
	as given by
	\begin{align*}
		(\dt - \Lap) \uppsi &= \ee^{\sigma t} \phi_R \bigg\{ 2 \|\A\|^2 \tr(\sind) - \frac{1}{2} \frac{\|\nabla \tr(\sind)\|^2}{\tr(\sind)} \\
				& \qquad \qquad \qquad + \frac{2}{\tr(\sind)} \sum_{k=1}^2 \left( \frac{2\sv_2}{1+\sv_2^2} \A_{1k}^1 + \frac{2\sv_1}{1+\sv_1^2} \A_{2k}^2 \right)^2 \bigg\} \\
			& \qquad - \ee^{\sigma t} \Bigl\{ (\Lap\phi_R) \tr(\sind) + 2 \< \nabla \phi_R, \nabla \tr(\sind) \> - \sigma \phi_R\tr(\sind) \Bigr\} \\
			& \eqqcolon \Ac + \Bc \,,
	\end{align*}
	where we collect all terms coming from the evolution equation of $\tr(\sind)$ (i.\,e.\ the first two lines)
	in $\Ac$ and the remaining terms (i.\,e.\ the third line) in $\Bc$.
	To estimate the terms in $\Ac$ at $(x_0,t_0)$,
	note that the vanishing of the first derivative in \eqref{eq:PsiCond} implies the equality
	\[
		(\nabla \phi_R)\tr(\sind) = - \phi_R \nabla \tr(\sind) \,.
	\]
	Consequently, since $\tr(\sind)\ge\frac{\varepsilon}{2}$ by assumption, at $(x_0,t_0)$ we derive the estimate
	\begin{align*}
		\Ac &\stackrel{\hphantom{\text{Lem.\ \ref{lem:phiEst}}}}{=} \ee^{\sigma t_0} \phi_R \bigg\{ \underbrace{2 \|\A\|^2 \tr(\sind)}_{\ge \|\A\|^2\varepsilon \ge 0} - \frac{1}{2} \frac{\|\nabla \tr(\sind)\|^2}{\tr(\sind)} \\
				& \qquad \qquad \qquad \qquad \qquad \qquad + \underbrace{\frac{2}{\tr(\sind)} \sum_{k=1}^2 \left( \frac{2\sv_2}{1+\sv_2^2} \A_{1k}^1 + \frac{2\sv_1}{1+\sv_1^2} \A_{2k}^2 \right)^2}_{\ge 0} \bigg\} \\
			& \stackrel{\hphantom{\text{Lem.\ \ref{lem:phiEst}}}}{\ge} - \frac{\ee^{\sigma t_0} \phi_R}{2} \frac{\|\nabla\tr(\sind)\|^2}{\tr(\sind)} = \frac{\ee^{\sigma t_0}}{2} \< \nabla \phi_R, \nabla \tr(\sind) \> \\
			& \stackrel{\text{Lem.\ \ref{lem:phiEst}}}{\ge} - \frac{\ee^{\sigma t_0}}{2} \frac{\|x_0\|_{\Rb^2}}{R^2} c(T') \tr(\sind) \,.
	\end{align*}
	Lemma \ref{lem:phiEst} and further evaluation yields
	\begin{align*}
		\Ac + \Bc &\ge - \frac{\ee^{\sigma t_0}}{2} \frac{\|x_0\|_{\Rb^2}}{R^2} c(T') \tr(\sind) \\
			& \quad - \ee^{\sigma t_0}\left\{ c(T') \left( \frac{1}{R^2} + \frac{\|x_0\|_{\Rb^2}}{R^2} \right) + 2 c(T') \frac{\|x_0\|_{\Rb^2}}{R^2} \right. \\
				& \qquad \qquad \qquad \qquad \qquad\qquad\qquad  \left.{} - \sigma - \sigma \frac{\|x_0\|^2_{\Rb^2}}{R^2} \right\} \tr(\sind) \\
			&= \ee^{\sigma t_0}\left\{ \sigma + \sigma \frac{\|x_0\|^2_{\Rb^2}}{R^2} - \frac{7}{2} c(T') \frac{\|x_0\|_{\Rb^2}}{R^2} - \frac{c(T')}{R^2} \right\} \tr(\sind) \,.
	\end{align*}
	Note that by choosing $R_0>0$ (depending on $\sigma$ and $T'$) large enough, the term
	\[
		\frac{\sigma}{2} + \sigma \frac{\|x_0\|^2_{\Rb^2}}{R^2} - \frac{7}{2} c(T') \frac{\|x_0\|_{\Rb^2}}{R^2} - \frac{c(T')}{R^2}
	\]
	is strictly positive for any $R\ge R_0$ and any $\|x_0\|_{\Rb^2}$. Continuing with the above
	calculation, we obtain
	\[
		\left( \partial_t - \Lap \right) \uppsi_{|(x_0,t_0)} \ge \ee^{\sigma t_0} \frac{\sigma}{2} \tr(\sind) \ge \ee^{\sigma t_0} \frac{\sigma}{2} \frac{\varepsilon}{2} > 0 \,.
	\]
	But this
	is a contradiction to \eqref{eq:PsiCond}, which shows the claim.
	
	The statement of the Lemma follows by first letting $R\to\infty$, then $\sigma\to 0$ and finally
	$T'\to T$.
\end{proof}

\begin{lemma}
	\label{lem:AreaDecrPres}
	Let $F(x,t)$ be a smooth solution to \eqref{eq:MCF} for $t\in[0,T)$ with bounded geometry.
	If there is $\varepsilon>0$ with $\tr(\sind)\ge\varepsilon$ at $t=0$, then
	$\tr(\sind)\ge\varepsilon$ for all $t\in[0,T)$.
\end{lemma}

\begin{proof}
	By Lemma \ref{lem:AreaDecrEpsilon}, we only need to remove the assumption $\tr(\sind)\ge\frac{\varepsilon}{2}$
	in $[0,T)$. By the bounded geometry assumption on $F(x,t)$, the right
	hand side of the evolution equation of $\tr(\sind)$ is bounded, so that
	\[
		\|\partial_t \tr(\sind)\| \le C(t) \,,
	\]
	where $C(t)$ is a constant only depending on $t$. Since $\tr(\sind)\ge \varepsilon$ at
	$t=0$, it follows that there is a maximal time $T_0>0$, such that $\tr(\sind)>\frac{\varepsilon}{2}$
	holds in $[0,T_0)$. From Lemma \ref{lem:AreaDecrEpsilon} we know that $\tr(\sind)\ge\varepsilon$
	on $\Rb^2\times[0,T_0)$. If $T_0\ne T$, by continuity, we also know that $\tr(\sind)\ge \varepsilon$
	on $\Rb^2\times[0,T_0]$. By the same argument for finding $T_0$ above, we can find
	some positive $T_0'$, such that $\tr(\sind)\ge \frac{\varepsilon}{2}$ in $\Rb^2\times[T_0,T_0+T_0')$,
	where $[T_0,T_0+T_0')\subset [T_0,T)$. But this contradicts the choice of $T_0$, so that $T_0=T$.
\end{proof}

\begin{lemma}
	\label{lem:GraphPres}
	Let $f:\Rb^2\to\Rb^2$ be a smooth strictly area-decreasing map and evolve it by
	the mean curvature flow. Then each $F_t(\Rb^2)$ is the graph of a strictly area-decreasing
	map for $t\in[0,T)$.
\end{lemma}

\begin{proof}
	The proof is the same as \cite{SHS14}*{Proof of Proposition 3.3}.
\end{proof}

\subsection{A priori estimates}
\label{sec:APriori}

To obtain estimates for the mean curvature vector, let us define the function
\[
	\upchi(x,t) \coloneqq \ee^{\sigma t}\phi_R(x) \tr(\sind)_{|(x,t)} - \varepsilon_2 \bigl( t \|\Hv\|^2_{|(x,t)} + 1 \bigr) \,.
\]

\begin{lemma}
	\label{lem:chiEvol}
	The evolution equation for $\upchi$ under the mean curvature flow is given by
	\begin{align*}
		\bigl( \partial_t - \Lap \bigr) \upchi &= \ee^{\sigma t}\phi_R \left\{ 2 \|\A\|^2\tr(\sind) + \frac{2}{\tr(\sind)} \sum_{k=1}^2 \left( \frac{2\sv_2}{1+\sv_2^2} \A_{1k}^1 + \frac{2\sv_1}{1+\sv_1^2} \A_{2k}^2 \right)^2 \right\} \\
			& \quad - \frac{1}{2} \ee^{\sigma t} \phi_R \frac{\|\nabla \tr(\sind)\|^2}{\tr(\sind)} \\
			& \quad - \varepsilon_2 t \left\{ - 2 \|\nN \Hv\|^2 + 2 \sum_{i,j=1}^2 \A_{\Hv}^2(e_i,e_j) \right\} - \varepsilon_2 \|\Hv\|^2 \\
			& \quad + \ee^{\sigma t} \Bigl\{ \sigma \phi_R \tr(\sind) - (\Lap \phi_R)\tr(\sind) - 2 \<\nabla \phi_R, \nabla \tr(\sind)\> \Bigr\} \,.
	\end{align*}
\end{lemma}

\begin{proof}
	We calculate
	\begin{align*}
		\bigl( \partial_t - \Lap \bigr) \upchi &= \ee^{\sigma t}\phi_R\bigl( \partial_t - \Lap \bigr) \tr(\sind) - \varepsilon_2 t \bigl( \partial_t - \Lap \bigr) \|\Hv\|^2 \\
			&\quad - \varepsilon_2 \|\Hv\|^2 \\
			&\quad + \ee^{\sigma t} \Bigl\{ \sigma \phi_R \tr(\sind) - (\Lap\phi_R)\tr(\sind) - 2 \<\nabla\phi_R,\nabla \tr(\sind)\> \Bigr\} \,.
	\end{align*}
	Now, recall (see e.\,g.\ \cite{Smo12}*{Corollary 3.8}) that the square norm of the mean curvature vector evolves by
	\begin{align*}
		\bigl( \partial_t - \Lap \bigr) \|\Hv\|^2 &= -2 \|\nN \Hv\|^2 + 2 \sum_{i,j=1}^2 \A_{\Hv}^2(e_i,e_j) \,,
	\end{align*}
	which together with Lemma \ref{lem:TrEvol} implies the claim.
\end{proof}

\begin{lemma}
	\label{lem:HDecay}
	Let $F(x,t)$ be a smooth, graphic solution to \eqref{eq:MCF} with bounded geometry and
	suppose $\tr(\sind)\ge \varepsilon_1$ on $[0,T)$ for some $\varepsilon_1>0$. Then there
	is a constant $C\ge 0$ depending only on $\varepsilon_1$, such that
	\[
		t \|\Hv\|^2 \le C
	\]
	on $\Rb^2\times[0,T)$.
\end{lemma}

\begin{proof}
	Fix $0<\varepsilon_2 < \varepsilon_1$, so that $\upchi$ is positive 
	on $\Rb^2\times\{0\}$. Further, fix any $T'\in[0,T)$. We will first show that we can choose $R_0>0$, such that $\upchi\ge 0$ on $\Rb^2\times[0,T']$
	for all $R\ge R_0$.
	
	Suppose $\upchi$ is not positive on $\Rb^2\times[0,T']$ for some $R\ge R_0$. Then, as $\upchi>0$ on
	$\Rb^2\times\{0\}$, $\tr(\sind)\ge\varepsilon_1$ on $[0,T)$, $\phi_R(x)\to\infty$ as $\|x\|\to\infty$ and
	by the bounded geometry condition \eqref{eq:BddGeom1}, it follows that $\upchi>0$ outside some compact set
	$K\subset\Rb^2$ for all $t\in[0,T']$. We conclude that there is $(x_0,t_0)\in K\times[0,T']$, such that
	$\upchi(x_0,t_0)=0$ and that $t_0$ is the first such time. By the second derivative criterion, at $(x_0,t_0)$
	we have
	\begin{equation}
		\label{eq:ChiSDC}
		\upchi = 0 \,, \quad \nabla \upchi = 0 \,, \quad \partial_t \upchi \le 0 \quad \text{and} \quad \Lap \upchi \ge 0 \,.
	\end{equation}
	On the other hand, we estimate the terms in the evolution equation for $\upchi$ from Lemma \ref{lem:chiEvol}
	at $(x_0,t_0)$.
	Using
	\[
		\sum_{i,j=1}^2 \A^2_{\Hv}(e_i,e_j) \le \|\A\|^2 \|\Hv\|^2
	\]
	and $\phi_R\ge 1$ yields the estimate
	\begin{align*}
		\bigl( \partial_t - \Lap \bigr) \upchi &\ge \ee^{\sigma t_0}\phi_R \Biggl\{ 2 \|\A\|^2 \tr(\sind) + \underbrace{\frac{2}{\tr(\sind)} \sum_{k=1}^2 \left( \frac{2\sv_2}{1+\sv_2^2} \A_{1k}^1 + \frac{2\sv_1}{1+\sv_1^2} \A_{2k}^2 \right)^2}_{\ge 0} \Biggr\} \\
			& \quad - \frac{1}{2} \ee^{\sigma t_0} \phi_R \frac{\|\nabla \tr(\sind)\|^2}{\tr(\sind)} \\
			& \quad + 2 \varepsilon_2 t_0 \bigl\| \nN \Hv\bigr\|^2 - 2 \varepsilon_2 t_0 \|\A\|^2 \|\Hv\|^2 - \varepsilon_2 \|\Hv\|^2 \\
			& \quad + \ee^{\sigma t_0}\Bigl\{ \sigma \phi_R \tr(\sind) - (\Lap \phi_R)\tr(\sind) - 2 \<\nabla \phi_R, \nabla\tr(\sind)\> \Bigr\} \\
		& \ge 2 \|\A\|^2 \left( \ee^{\sigma t_0} \phi_R \tr(\sind) - \varepsilon_2 \bigl( t_0 \|\Hv\|^2 + 1 \bigr) \right) + 2 \varepsilon_2 \|\A\|^2 - \varepsilon_2 \|\Hv\|^2 \\
			& \quad - \frac{1}{2} \ee^{\sigma t_0}\phi_R \frac{\|\nabla \tr(\sind)\|^2}{\tr(\sind)} + 2 \varepsilon_2 t_0 \|\nN\Hv\|^2 \\
			& \quad + \ee^{\sigma t_0}\Bigl\{ \sigma \phi_R \tr(\sind) - (\Lap \phi_R)\tr(\sind) - 2 \<\nabla \phi_R, \nabla\tr(\sind)\> \Bigr\} \\
		& \eqqcolon 2 \|\A\|^2 \upchi + \Ac + \mathcal{G} \\
			& \quad + \ee^{\sigma t_0}\Bigl\{ \sigma \phi_R \tr(\sind) - (\Lap \phi_R)\tr(\sind) - 2 \<\nabla \phi_R, \nabla\tr(\sind)\> \Bigr\} \,,
	\end{align*}
	where
	\begin{align*}
		\Ac &\coloneqq 2 \varepsilon_2 \|\A\|^2 - \varepsilon_2 \|\Hv\|^2 \,, \\
		\mathcal{G} &\coloneqq - \frac{1}{2} \ee^{\sigma t_0}\phi_R \frac{\|\nabla \tr(\sind)\|^2}{\tr(\sind)} + 2 \varepsilon_2 t_0 \|\nN\Hv\|^2 \,.
	\end{align*}
	Since $\|\Hv\|^2\le 2 \|\A\|^2$, we derive 
	\begin{equation}
		\label{eq:AEst}
		\Ac\ge 0 \,.
	\end{equation}
	To estimate the terms in $\mathcal{G}$, we want to exploit
	$\nabla\upchi=0$ at $(x_0,t_0)$. This yields
	\[
		\ee^{\sigma t_0} \bigl\{ (\nabla\phi_R) \tr(\sind) + \phi_R (\nabla\tr(\sind)) \bigr\} = \varepsilon_2 t_0 \nabla \| \Hv\|^2
	\]
	and consequently
	\begin{align*}
		\ee^{2\sigma t_0} \bigl\| (\nabla \phi_R)\tr(\sind) + \phi_R(\nabla \tr(\sind)) \bigr\|^2 &= \varepsilon_2^2 t_0^2 \bigl\| \nabla \|\Hv\|^2 \bigr\|^2 \\
			& \le 4 \varepsilon_2^2 t_0^2 \|\nN\Hv\|^2 \|\Hv\|^2 \\
			& \le 4 \varepsilon_2^2 t_0^2 \|\nN\Hv\|^2 \bigl( \|\Hv\|^2+1 \bigr) \,.
	\end{align*}
	From $\upchi(x_0,t_0)=0$ we get $ \ee^{\sigma t_0} \phi_R \tr(\sind) = \varepsilon_2 t_0 \bigl( \|\Hv\|^2 + 1 \bigr)$, 
	so that
	\[
		\ee^{\sigma t_0} \bigl\| (\nabla \phi_R)\tr(\sind) + \phi_R(\nabla \tr(\sind)) \bigr\|^2 \le 4 \varepsilon_2 t_0 \phi_R \|\nN\Hv\|^2 \tr(\sind) \,.
	\]
	Noting $\phi_R\ge 1$ and sorting the expression, we obtain
	\begin{align*}
		\ee^{\sigma t_0} \phi_R \|\nabla\tr(\sind)\|^2 &\le 4 \varepsilon_2 t_0 \|\nN\Hv\|^2 \tr(\sind) \\
			& \quad - \ee^{\sigma t_0} \left\{ \frac{\|\nabla\phi_R\|^2}{\phi_R} \bigl(\tr(\sind)\bigr)^2 + 2 \tr(\sind) \< \nabla \phi_R, \nabla \tr(\sind) \> \right\} \\
			& \le 4 \varepsilon_2 t_0 \|\nN\Hv\|^2 \tr(\sind) - 2 \ee^{\sigma t_0} \tr(\sind) \<\nabla \phi_R, \nabla\tr(\sind)\> \,.
	\end{align*}
	Thus, the gradient terms satisfy
	\begin{equation}
		\label{eq:GEst}
		\mathcal{G} = - \frac{1}{2} \ee^{\sigma t_0} \phi_R \frac{\|\nabla\tr(\sind)\|^2}{\tr(\sind)} + 2 \varepsilon_2 t_0 \|\nN \Hv\|^2 \ge \ee^{\sigma t_0} \< \nabla \phi_R, \nabla \tr(\sind) \> \,.
	\end{equation}
	Collecting the previous calculations and using $\tr(\sind)\ge\varepsilon_1>0$ as well as $\upchi(x_0,t_0)=0$,
	we estimate the evolution equation of $\upchi$ at $(x_0,t_0)$ by
	\begin{align*}
		(\dt - \Lap) \upchi &\stackrel{\text{Eqs.\ \eqref{eq:AEst}, \eqref{eq:GEst}}}{\ge} \ee^{\sigma t_0} \Bigl\{ \sigma \phi_R \tr(\sind) - (\Lap \phi_R)\tr(\sind) - \< \nabla \phi_R, \nabla \tr(\sind)\> \Bigr\} \\
			& \stackrel{\substack{\hphantom{\text{Eqs.\ \eqref{eq:AEst}, \eqref{eq:GEst}}}\\\text{Lem.\ \ref{lem:phiEst}}}}{\ge} \ee^{\sigma t_0} \left\{ \sigma\left(1 + \frac{\|x_0\|^2_{\Rb^2}}{R^2} \right) - c(T') \left(\frac{1}{R^2} + \frac{\|x_0\|_{\Rb^2}}{R^2} \right) \right. \\
				& \qquad \qquad \qquad \qquad \qquad \qquad \qquad \qquad \qquad \left. - c(T') \frac{\|x_0\|_{\Rb^2}}{R^2} \right\} \tr(\sind) \\
			& \stackrel{\hphantom{\text{Eqs.\ \eqref{eq:AEst}, \eqref{eq:GEst}}}}{=} \!\!\! \ee^{\sigma t_0} \left\{ \sigma + \sigma \frac{\|x_0\|^2_{\Rb^2}}{R^2} - 2 c(T') \frac{\|x_0\|_{\Rb^2}}{R^2} - \frac{c(T')}{R^2} \right\} \tr(\sind) \,.
	\end{align*}
	Now we choose $R_0>0$ (depending on $\sigma$ and $T'$) large enough, so that the term
	\[
		\frac{\sigma}{2} + \sigma \frac{\|x_0\|^2_{\Rb^2}}{R^2} - 2 c(T') \frac{\|x_0\|_{\Rb^2}}{R^2} - \frac{c(T')}{R^2}
	\]
	is strictly positive for any $R\ge R_0$ and any $\|x_0\|_{\Rb^2}$. Continuing with the above
	calculation, we obtain
	\[
		(\dt - \Lap) \upchi_{|(x_0,t_0)} \ge \ee^{\sigma t_0} \frac{\sigma}{2} \tr(\sind) \ge \ee^{\sigma t_0} \frac{\sigma}{2} \varepsilon_1 > 0 \,.
	\]
	But this is a contradiction to \eqref{eq:ChiSDC}, which shows the claim.
	
	By first letting $R\to\infty$, then $\sigma\to 0$ and finally $T'\to T$, we have shown
	that
	\[
		\tr(\sind) - \varepsilon_2 \bigl( t\|\Hv\|^2 + 1 \bigr) \ge 0
	\]
	holds for all $t\in[0,T)$.
	The statement of the Lemma follows by noting $\tr(\sind)\le 2$,
	setting $C\coloneqq \frac{2}{\varepsilon_2} - 1$ and recalling that $\varepsilon_2$ only depends
	on $\varepsilon_1$.
\end{proof}

As in \cites{CCH12,Lub16}, we go on by analyzing the non-parametric version of the mean curvature flow
to obtain estimates on all higher derivatives of the map which defines the graph. Note that
most proofs are very similar to the ones in the articles cited, but nevertheless
need to be slightly modified to account for the weaker assumptions in the two-dimensional case.

\begin{lemma}
	\label{lem:HigherOrderEst1}
	Let $F:\Rb^2\times[0,T)\to\Rb^2\times\Rb^2$ be a smooth, graphic solution to \eqref{eq:MCF}
	with bounded geometry. Suppose the corresponding maps $f_t:\Rb^2\to\Rb^2$ satisfy $\|\D f_t\|\le C_1$
	and $\|\D^2 f_t\|\le C_2$ on $\Rb^2\times[0,T)$ for some constants $C_1,C_2\ge 0$. Then
	for every $l\ge 3$, there is a constant $C_l$, such that
	\[
		\sup_{x\in\Rb^2} \|\D^l f_t(x)\|^2 \le C_l
	\]
	for all $t\in[0,T)$.
\end{lemma}

\begin{proof}
	The proof is essentially the same as \cite{CCH12}*{Proof of Lemma 4.2} (see also \cite{Lub16}*{Lemma 5.4} with $m=n=2$).
\end{proof}

\begin{lemma}
	\label{lem:HigherOrderEst2}
	Let $F:\Rb^2\times[0,T)\to\Rb^2\times\Rb^2$ be a smooth, graphic solution to \eqref{eq:MCF} with
	bounded geometry and denote by $f_t:\Rb^2\to\Rb^2$ the corresponding maps. Assume the condition
	$\tr(\sind)\ge \varepsilon$ holds for a fixed $\varepsilon>0$ at time $t=0$. Further assume that
	$\|\Hv\|\le C$ on $\Rb^2\times[0,T)$ for some constant $C\ge 0$. Then for every $l\ge 1$, there is
	a constant $C_l\ge 0$, such that
	\[
		\sup_{x\in\Rb^2} \|\D^l f_t(x)\|^2 \le C_l
	\]
	for all $t\in [0,T)$.
\end{lemma}

\begin{proof}
	By Lemma \ref{lem:AreaDecrPres}, the area-decreasing condition is preserved in $[0,T)$, so that the relation
	$\tr(\sind)\ge \varepsilon$ holds in $[0,T)$. Since $\varepsilon$ is strictly positive, from
	\[
		\tr(\sind) = \frac{2(1-\sv_1^2\sv_2^2)}{(1+\sv_1^2)(1+\sv_2^2)} \ge \varepsilon > 0
	\]
	we infer
	\begin{equation}
		\label{eq:SVBound}
		\varepsilon (1 + \sv_i^2) \le \frac{2(1-\sv_1^2\sv_2^2)}{1+\sv_j^2} \le 2 \,, \qquad (i,j) \in\bigl\{ (1,2), (2,1) \bigr\} \,,
	\end{equation}
	so that the singular values $\sv_1,\sv_2$ of $\D f_t$ are bounded. This also means that $\D f_t$
	itself is bounded, thus showing the claim for $l=1$.
	
	By Lemma \ref{lem:HigherOrderEst1}, we now only need to prove the case $l=2$. Suppose the claim was false
	for $l=2$. Let
	\[
		\eta(t) \coloneqq \sup_{\substack{x\in\Rb^2\\t'\le t}} \|\D^2 f(x,t')\| \,.
	\]
	Then there is a sequence $(x_k,t_k)$ along which we have $\|\D^2 f(x_k,t_k)\|\ge \eta(t_k)/2$ while
	$\eta(t_k)\to\infty$ as $t_k\to T$. Let $\tau_k\coloneqq \eta(t_k)$. For each $k$, let $(y,\fys{\tau_k})$
	be the parabolic scaling of the graph $(x,f(x,t))$ by $\tau_k$ at $(x_k,t_k)$. Then $\fys{\tau_k}$ is a
	smooth solution to \eqref{eq:pMCF} on $\Rb^2\times[-\tau_k^2 t_k,0]$. Note that by the definition $\tau_k=\eta(t_k)$,
	it is
	\begin{align*}
		\|\Dy \fScaled{\tau_k}\| &= \|\D f\| \le C_1 \,, \\
		\|\Dy^2 \fScaled{\tau_k}\| &= \tau_k^{-1} \|\D^2 f\| \le 1
	\end{align*}
	on $\Rb^2\times[-\tau_k^2 t_k,0]$. Moreover, by the definition of the sequence $(x_k,t_k)$, the estimate
	\begin{equation}
		\label{eq:fScaleEst1}
		\|\Dy^2 \fScaled{\tau_k}(0,0)\| = \frac{\| \D^2 f(x_k,t_k)\|}{\tau_k} = \frac{ \|\D^2 f(x_k,t_k)\|}{\eta(t_k)} \ge \frac{1}{2}
	\end{equation}
	holds. By Lemma \ref{lem:HigherOrderEst1}, we conclude that all the higher derivatives of $\fScaled{\tau_k}$
	are uniformly bounded on $\Rb^2\times[-\tau_k^2 t_k,0]$. Thus, the theorem of Arzel\`{a}-Ascoli implies
	the existence of a subsequence of $\fScaled{\tau_k}$ converging smoothly and uniformly on compact subsets
	of $\Rb^2\times(-\infty,0]$ to a smooth solution $\fScaled{\infty}$ to \eqref{eq:pMCF}. Since
	$\|\Hv\|\le C$ for the graphs $\bigl(x,f(x,t)\bigr)$ by assumption, after rescaling
	we have
	\[
		\bigl\|\Hv_{\tau_k}\bigr\| \le \frac{C}{\tau_k}
	\]
	for the graphs $\bigl( \xScaled, \fys{\tau_k} \bigr)$. It follows that for each $\tScaled$ the limiting
	graph $\bigl( \xScaled, \fys{\infty} \bigr)$ must have $\|\Hv_{\infty}\|=0$ everywhere, as well as
	$\tr(\sind)\ge \varepsilon$. Note that by Equation \eqref{eq:SVBound}, this implies bounds
	for the singular values $\widetilde{\sv}_1,\widetilde{\sv}_2$ of the limiting graph,
	\[
		1+\widetilde{\sv}_k^2 \le \frac{2}{\varepsilon} \,, \qquad k=1,2 \,.
	\]
	It follows that we can estimate the Jacobian
	of the projection $\RpiM$ from the graph $\bigl( \xScaled, \fys{\infty} \bigr)$ to $\Rb^2$,
	\[
		0 < \frac{\varepsilon}{2} \le \star \Omega_{\infty} = \frac{1}{\sqrt{(1+\widetilde{\sv}_1^2)(1+\widetilde{\sv}_2^2)}} \le 1 \,.
	\]
	Thus, we can apply a Bernstein-type theorem of Wang \cite{Wan03a}*{Theorem 1.1} to conclude
	that the graph $\bigl(y,\fys{\infty}\bigr)$ is an affine subspace of $\Rb^2\times\Rb^2$. Therefore,
	$\fys{\infty}$ has to be a linear map, but this contradicts \eqref{eq:fScaleEst1}, which (taking the
	limit $k\to\infty$) implies the estimate $\|\Dy^2 \fys{\infty}(0,0)\| \ge 1/2$.
\end{proof}

\begin{lemma}
	\label{lem:HeightEst}
	Suppose $f(x,t)$ is a smooth solution to \eqref{eq:pMCF} on $[0,T)$ that satisfies the bounded geometry
	condition. Then
	\[
		\sup_{x\in\Rb^2} \| f(x,t)\|^2 \le \sup_{x\in\Rb^2} \|f(x,0)\|^2
	\]
	holds for all $t\in [0,T']$, where $T'\in[0,T)$ is arbitrary.
\end{lemma}

\begin{proof}
	This is \cite{Lub16}*{Lemma 5.6} with $m=n=2$.
\end{proof}

\section{Proof of Theorem \ref{thm:ThmA}}
\label{sec:Proof}

Using Lemma \ref{lem:AreaDecrPres} and the estimates from 
the Lemmas \ref{lem:HDecay}, \ref{lem:HigherOrderEst1} and \ref{lem:HigherOrderEst2},
the proof of the long-time existence of the solution
is the same as in \cite{CCH12}*{Lemma 5.2}.
By Lemma \ref{lem:GraphPres}, the evolving surface stays a graph of an
area-decreasing map $f_t:\Rb^2\to\Rb^2$.
The decay estimate for $\Hv$ is given in Lemma \ref{lem:HDecay}.

Employing the Lemmas \ref{lem:AreaDecrPres},
the bounds on the singular values from Equation \eqref{eq:SVBound}
and the decay estimates from Lemmas \ref{lem:HDecay},
\ref{lem:HigherOrderEst1} and \ref{lem:HigherOrderEst2},
the proof of the decay estimates for the higher-order derivatives of $f_t$
follows in the same way as in \cite{Lub16}*{Lemma 6.3}.
The height estimate is provided by Lemma \ref{lem:HeightEst}. \\

If we assume $\|f_0\|\to 0$ for $\|x\|\to\infty$, we know by Lemma \ref{lem:HeightEst} that $\sup_{x\in\Rb^2}\|f(x,t)\|$
stays bounded. As the singular values $\sv_1,\sv_2\ge 0$ are uniformly bounded, so is $\widetilde{\gind}$, which
means the equation
\[
	\frac{\partial f}{\partial t}(x,t) = \sum_{i,j=1}^2 \widetilde{\gind}^{ij} \frac{\partial^2 f}{\partial x^i \partial x^j}(x,t)
\]
is uniformly parabolic. Then, by the theorem in \cite{Ili61}, $f(x,t)\to 0$ as $t\to\infty$, uniformly
with respect to $x$. This shows the convergence part of Theorem \ref{thm:ThmA} and concludes the proof. \hfill $\Box$

\section{Applications}
\label{sec:Appl}

We demonstrate how to apply Theorem \ref{thm:ThmA} to the examples considered in \cite{Lub16}*{Section 9}.
Note that both proofs are formally the same as \cite{Lub16}*{Proofs of Examples 9.3 and 9.4}, and
we state them here for completeness.

Let $F_t:\Rb^2\to\Rb^2\times\Rb^2$ be a graphical self-shrinking solution to the mean curvature flow, and
denote by $f_t:\Rb^2\to\Rb^2$ the corresponding map. Then $f_1$ satisfies the equation
\begin{equation}
	\label{eq:ShrSol}
	\sum_{i,j=1}^2 \tgind^{ij} \frac{\partial^2 f_1^k(x)}{\partial x^i \partial x^j} = - \frac{1}{2} f^k_1(x) + \frac{1}{2} \< \D f_1^k(x), x \> \,, \quad k=1,2 \,.
\end{equation}
If $F_t$ is a translating solution to the mean curvature flow, then there is $\xi\in\Rb^2\times\Rb^2$, such
that $\Hv = \prN(\xi)$. If the initial data is given by $F_0(x) = (x,f(x))$, then for $F_t$ to be a translating solution
the function $f$ has to satisfy
\begin{equation}
	\label{eq:TransSol}
	\sum_{i,j=1}^2 \tgind^{ij} \frac{\partial^2 f(x)}{\partial x^i \partial x^j} = \der \RpiN(\xi) - \< \D f(x), \der \RpiM(\xi) \> \,.
\end{equation}

\begin{example}[A Bernstein Theorem for Self-Shrinking Solutions]
	Let $v:\Rb^2\to\Rb^2$ be a strictly area-decreasing map with bounded second fundamental form and satisfying \eqref{eq:ShrSol}.
	Then $v$ is a linear function.
\end{example}

\begin{proof}
	Since $v$ is a smooth solution to \eqref{eq:ShrSol}, the function
	\[
		f_t(x) \coloneqq \sqrt{-t} v\left( \frac{x}{\sqrt{-t}} \right)
	\]
	is a solution to \eqref{eq:pMCF} for $t\in(-\infty,0]$ and $f_{-1}(x) = v(x)$. Since this solution is unique
	by \cite{CY07}*{Theorem 1.1}, we can apply Theorem \ref{thm:ThmA}. In particular, $\|\D^2 f_t(x)\|\le C$
	for some constant $C$ for $t\ge -1$ and any $x$. Since also
	\[
		\D^2 f_t(x) = \frac{1}{\sqrt{-t}} \D^2 v\left( \frac{x}{\sqrt{-t}} \right) \,,
	\]
	we obtain the estimate
	\[
		\|\D^2v(x)\| \le C\sqrt{-t}
	\]
	for any $x$. Letting $t\to 0$, this implies $\D^2v(x) = 0$, so that $v$ is a linear function.
\end{proof}

\begin{example}[A Bernstein Theorem for Translating Solutions]
	Let $v:\Rb^2\to\Rb^2$ be a strictly area-decreasing map with bounded second fundamental form and satisfying \eqref{eq:TransSol}.
	Then $v$ is a linear function.
\end{example}

\begin{proof}
	If $v$ solves \eqref{eq:TransSol}, then 
	there is a constant vector $\xi\in\Rb^2\times\Rb^2$, such that
	\[
		f_t(x) \coloneqq v\bigl( x - \der \RpiM(\xi)t \bigr) + \der \RpiN(\xi)t
	\]
	solves \eqref{eq:pMCF} with initial condition $f_0(x) = v(x)$.
	
	On the other hand, by Theorem \ref{thm:ThmA} there is a long-time solution $f_t(x)$ to \eqref{eq:pMCF} with
	initial condition $f_0$ which satisfies $\|\sup_{x\in\Rb^2} \D^2f_t(x) \| \to 0$ as $t\to\infty$. By
	the uniqueness result \cite{CY07}*{Theorem 1.1},
	\[
		\sup_{x\in\Rb^2} \bigl\| \D^2 v\bigl( x - \der \RpiM(\xi)t \bigr) \bigr\| = \sup_{x\in\Rb^2} \bigl\| \D^2 f_t(x) \bigr\| \to 0
	\]
	as $t\to\infty$. We conclude that $\sup_{x\in\Rb^2} \| \D^2v(x)\| = 0$, so $v$ must be linear.
\end{proof}

\vskip 5\baselineskip

\hypersetup{
	hidelinks
}

\begin{flushleft}	
	\textsc{Felix Lubbe} \\
	\textsc{Mathematisches Institut} \\
	\textsc{Georg-August-Universit\"{a}t G\"{o}ttingen} \\
	\textsc{Bunsenstra{\ss}e 3--5}\\
	\textsc{37073 G\"{o}ttingen, Germany} \\
	\textsl{E-mail address:} \textbf{\href{mailto:Felix.Lubbe@mathematik.uni-goettingen.de}{Felix.Lubbe@mathematik.uni-goettingen.de}}
\end{flushleft}

\end{document}